\documentclass{article}

\date{}

\usepackage{amssymb}
\usepackage{amsfonts}
\newtheorem{theorem}{{\bf Theorem}}[section]
\newtheorem{lemma}[theorem]{{\bf Lemma}}

\newtheorem{remark}[theorem]{{\bf Remark}}

\newtheorem{definition}[theorem]{{\bf Definition}}

\newtheorem{example}[theorem]{{\bf Example}}

\title{\Large \bf Topological groups of bounded homomorphisms\\ on a topological
group}

\author{{\bf Ljubi\v sa D.R. Ko\v cinac\footnote{Corresponding author}, Omid Zabeti}\\
University of Ni\v s, Faculty of Sciences and Mathematics,\\
18000 Ni\v s, Serbia\\
{\sf lkocinac@gmail.com}\\
Faculty of Mathematics, University of Sistan and Baluchestan,\\
P.O. Box 98135-674, Zahedan, Iran\\
{\sf o.zabeti@gmail.com}}

\begin{document}

\maketitle

\begin{abstract}
We consider a few types of bounded homomorphisms on a topological
group. These classes of bounded homomorphisms are, in a sense,
weaker than the class of continuous homomorphisms. We show that
with appropriate topologies each class of these homomorphisms on a
complete topological group forms a complete topological group.
\end{abstract}

\begin{flushleft}
{\sf 2010 Mathematics Subject Classification}: 22A05, 54H11. \\
\vspace{.3cm} {\sf Keywords}: Bounded homomorphism; continuous
homomorphism; topological group; completeness.
\end{flushleft}

%%%%%%%%%%%%%%%%%%%%%% 1111 %%%%%%%%%%%%
\section{Introduction}
%%%%%%%%%%%%%%%%% 1111 %%%%%%%%%%%%%%%%%

When one deals with the concept of a bounded set in a topological
vector space, there are some tools which are absolutely handy: for
example, the scalar multiplication and absorbing zero neighborhoods.
But when we want to consider the notion of a bounded set in a
topological group, the situation is completely different. There are
neither scalar field nor absorbing neighborhoods. One may define a
bounded subset in a topological group by taking group multiplication
instead of scalar multiplication in a topological vector space; but
this approach does not match our intuition for a bounded subset,
since, for example, the multiplicative  group $S^{1}$ is not bounded
in this definition.

\smallskip
Following \cite{Azar}, a subset $B$ of a topological group $G$ is
called \emph{bounded} if for each neighborhood $U$ of the identity
element $e_G$ of $G$ there is a positive integer $n$ such that
$B\subset U^n$.

\smallskip
We give here a result concerning this kind of boundedness in the
products of topological groups.

\begin{theorem}\label{intro}
Let $\{G_{\alpha}:\alpha\in\Lambda\}$ be a set of abelian
topological groups and $G=\prod_{\alpha \in \Lambda} G_{\alpha}$
with the product topology. Then $B\subseteq G$ is bounded if and
only if there are finitely many sets $B_{\alpha_1}, \ldots,
B_{\alpha_k}$ such that for each $i\le k$, $B_{\alpha_i}$ is a
bounded set in $G_{\alpha_i}$ and $B\subset \prod_{i\le
k}B_{\alpha_i} \times \prod_{\alpha \in \Lambda\setminus
\{\alpha_1,\ldots,\alpha_k\}}G_{\alpha}$.
\end{theorem}
\begin{proof} Suppose $B\subseteq G$ is bounded. Put
\[
B_{\alpha_i}=\{x\in G_{\alpha_i}: \exists \ \mathbf{y} =
(y_{\alpha})\in B \mbox{ and $x$ is $\alpha_i$-th coordinate of
$\mathbf{y}$}\}.
\]
Each $B_{\alpha_i}$, $i\le k$, is bounded. For, if $U_i$, $i\le k$,
is a neighborhood of the identity in $G_{\alpha_i}$ put
\[
U= U_{\alpha_1} \times \ldots \times U_{\alpha_k} \times
\prod_{\alpha \neq \alpha_i}G_{\alpha}.
\]
$U$ is a neighborhood of the identity in $G$. Therefore, there is a
positive integer $n$ with $B\subseteq U^n$, so that
$B_{\alpha_i}\subseteq U_i^n$. It is easy to see that $B\subseteq
B_{\alpha_1}\times \ldots B_{\alpha_k} \times \prod_{\alpha \in
\Lambda \setminus \{\alpha_1,\ldots, \alpha_k\}}G_{\alpha}$.

\smallskip
Suppose now that for a set $B\subseteq G$, there exist sets
$B_{\alpha_1}, \ldots, B_{\alpha_k}$ such that $B_{\alpha_i}$ is a
bounded set in $G_{\alpha_i}$, $i\le k$, and $B\subset \prod_{i\le
k}B_{\alpha_i} \times \prod_{\alpha \in \Lambda\setminus
\{\alpha_1,\ldots,\alpha_k\}}G_{\alpha}$. Let $U$ be an arbitrary
neighborhood of the identity element of $G$; without loss of
generality, we may assume that $U$ is of the form
\[
U=U_{1}\times\ldots \times U_{k}\times \prod_{\alpha \in
\Lambda\setminus \{\alpha_1,\ldots,\alpha_k\}}G_{\alpha},
\]
where $U_{i}$ is a neighborhood of the identity element in
$G_{\alpha_i}$, $i\le k$. For each $i\le k$ there is a positive
integer $n_i$ such that $B_{\alpha_i}\subset U_i^{n_i}$. Put
$n=\max\{n_1,\ldots,n_k\}$. Evidently $B\subset U^n$, i.e. $B$ is
bounded in $G$.
\end{proof}

\medskip
Bounded homomorphisms and their algebraic and topological structures
on any topological algebraic structure are of interest for their own
rights and also for their applications in other areas of
mathematics. For examples, bounded operators on a topological vector
space with suitable topologies form topological algebras; also,
there is a spectral theory for these classes of bounded operators
with some useful applications (see \cite{Omid 1, Tr, Omid 2} for
ample information). Therefore, it will be of interest to consider
different types of bounded homomorphisms on a topological group and
to investigate which topological and algebraic properties of the
underlying topological group can be transferred to the mentioned
classes of homomorphisms. In this note, we examine a few notions of
bounded homomorphisms on a topological group. These classes of
bounded homomorphisms contain continuous homomorphisms. We equip
each class of such bounded homomorphisms by an appropriate topology.
It turns out that they constitute complete topological groups
provided that the underlying topological group is assumed to be
complete and its singletons are bounded.

For terminology and notations used in this paper, we refer the
reader to \cite{Hewitt}, as well as for an abstract and
comprehensive taste of topological groups and related notions. All
topological groups in this note are assumed to be \emph{abelian}. It
is well known that a topological group has a local base at the
identity element consisting of symmetric neighborhoods. Throughout
the paper we consider only such bases. The identity element of a
group $G$ will be denoted by $e_G$.

%%%%%%%%%%%%%%%%%%%%%% 2222 %%%%%%%%%%%%%%%%
\section{Results}
%%%%%%%%%%%%%%%%%%%%%% 2222 %%%%%%%%%%%%%%%%

\begin{definition}\rm
Let $G$ and $H$ be two topological groups. A homomorphism $T:G \to
H$ is said to be
\begin{itemize}
\item[$(1)$] \emph{{\sf nb}-bounded} if there exists a
neighborhood $U$ of $e_G$ such that $T(U)$ is bounded in $H$;

\item[$(2)$] \emph{{\sf bb}-bounded} if for every bounded set $B
\subset G$, $T(B)$ is bounded in $H$.
\end{itemize}
\end{definition}

The set of all {\sf nb}-bounded ({\sf bb}-bounded) homomorphisms
from a topological group $G$ to a topological group $H$ is denoted
by ${\sf Hom_{nb}}(G,H)$ (${\sf Hom_{bb}}(G,H)$). We write ${\sf
Hom}(G)$ instead of ${\sf Hom}(G,G)$.

\begin{remark}\rm For topological groups $G$ and $H$ the following holds:
\[
{\sf Hom_{nb}}(G,H) \subset {\sf Hom_{bb}}(G,H).
\]

\smallskip
Let $T:G\to H$ be an {\sf nb}-bounded homomorphism. Then it is {\sf
bb}-bounded. For, suppose $B\subset G$ is a bounded set. Since $T$
is ${\sf nb}$-bounded there is a neighborhood $U$ of $e_G$ such that
$T(U)$ is bounded in $H$. Boundedness of $B$ implies $B\subset U^n$
for some natural number $n$. We prove that $T(B)$ is bounded in $H$.
Let $V$ be a neighborhood of $e_H$. Boundedness of $T(U)$ implies
that there is $m\in\mathbb N$ such that $T(U) \subset V^m$. Then
\[
T(B)\subset T(U^n)=(T(U))^n \subset V^{mn},
\]
i.e. $T(B)$ is bounded in $H$.
\end{remark}

Note that the converse is not true as the following example shows.

\begin{example}\rm \label{ex1}
Let $G=\mathbb C-\{0\}$ be the group of non-zero complex numbers,
and $H=G^{\mathbb N}$, the group of all sequences of elements of
$G$, with the pointwise multiplication and the product topology.
Consider the identity homomorphism $1_H$ on $H$. It is easy to see
that $1_H$ is ${\sf bb}$-bounded, but it is not ${\sf nb}$-bounded
since $H$ is not locally bounded.
\end{example}

\begin{remark} \rm
One may verify that each continuous homomorphism is ${\sf
bb}$-bounded. Let $G$ and $H$ be topological groups, $T:G\to H$ be a
continuous homomorphism, and $B$ a bounded subset of $G$. Suppose a
neighborhood $V$ of $e_H$ is given. There exists a neighborhood $U$
of $e_G$ such that $T(U)\subset V$. Also, since $B$ is bounded in
$G$, there is an $n\in \mathbb N$ with $B\subset U ^n$. Thus,
\[
T(B)\subset T(U^n)=(T(U))^n\subset V^n.
\]
Nevertheless, unlike in the case of bounded operators on topological
vector spaces, there is no more relation between continuous
homomorphisms on a topological group and bounded ones, which is
explained in the example below.
\end{remark}

\begin{example} \rm
Let $G$ be $S^1$, the multiplicative group of all complex numbers of
modulus 1, with the trivial (anti-discrete) topology, and $H$ be
$S^1$ with the topology inherited from $\mathbb C$. Consider the
identity homomorphism $i$ from $G$ into $H$. Clearly, $i$ is not
continuous. But it is {\sf nb}-bounded, as well as {\sf bb}-bounded.
\end{example}

\begin{remark} \rm
Note that if $G$ is a topological vector space over $\mathbb R$ or
$\mathbb C$, we have two notions for a bounded set; one of them is
related to the group structure of $G$ and another considers the
scalar multiplication. It is easy to see that these aspects of
bounded sets have the same meaning.
\end{remark}

Let $G$ be a group. For homomorphisms $T$ and $S$ on $G$, define
$TS$ and $T^{-1}$ by
\[
 TS(x):=T(x)S(x) \mbox{ and } T^{-1}(x):=(T(x))^{-1}, \, \, \,
 x\in G.
\]
It is easy to see that with these operations, the class of all
homomorphisms on a  group $G$ forms a group.

\smallskip
Now, assume $G$ is a topological group. The class of all ${\sf
nb}$-bounded homomorphisms on $G$ equipped with the topology of
uniform convergence on some neighborhood of $e_G$ is denoted by
${\sf Hom_{nb}}(G)$. Observe that a net $(S_{\alpha})$ of ${\sf
nb}$-bounded homomorphisms converges uniformly on a neighborhood $U$
of $e_G$ to a homomorphism $S$  if for each neighborhood $V$ of
$e_G$ there exists an $\alpha_0$ such that for each
$\alpha\geq\alpha_0$, $(S_{\alpha}S^{-1})(U)\subset V$.

\smallskip
The class of all ${\sf bb}$-bounded homomorphisms on $G$ endowed
with the topology of uniform convergence on bounded sets is denoted
by ${\sf Hom_{bb}}(G)$. Note that a net $(S_{\alpha})$ of ${\sf
bb}$-bounded homomorphisms uniformly converges to a homomorphism $S$
on a bounded set $B\subset G$ if for each neighborhood $V$ of $e_G$
there is an $\alpha_0$ with $(S_{\alpha}S^{-1})(B) \subset V$ for
each $\alpha\ge \alpha_0$.

\smallskip
The class of all continuous homomorphisms on $G$ equipped with the
topology of ${\sf c}$-convergence is denoted by ${\sf Hom_{c}}(G)$.
A net $(S_{\alpha})$ of continuous homomorphisms ${\sf c}$-converges
to a homomorphism $S$ if for each neighborhood $W$ of $e_G$, there
is a neighborhood $U$ of $e_G$ such that for every neighborhood $V$
of $e_G$ there exists an $\alpha_0$ with
$(S_{\alpha}S^{-1})(U)\subset VW$ for each $\alpha\geq\alpha_0$.

\smallskip
Note that ${\sf Hom_{c}}(G)$ and ${\sf Hom_{bb}}(G)$ form subgroups
of the group of all homomorphisms on $G$. On the other hand, ${\sf
Hom_{nb}}(G)$ is not a group, in general; the identity homomorphism
on $G$ in Example \ref{ex1} does not belong to ${\sf Hom_{nb}}(G)$.

\begin{theorem}
The operations of multiplication and inversion are continuous in
${\sf Hom_{nb}}(G)$ $($with respect to the topology of uniform
convergence on some neighborhood of $e_G${\rm )}.
\end{theorem}
\begin{proof}
Suppose $(T_{\alpha})$ and $(S_{\alpha})$ are two nets of ${\sf
nb}$-bounded homomorphisms which are convergent uniformly on some
neighborhood $U$ of $e_G$ to the ${\sf nb}$-bounded homomorphisms
$T$ and $S$, respectively. Let $W$ be an arbitrary neighborhood of
$e_G$. There is a neighborhood $V$ of $e_G$ with $VV\subset W$.
There exist some $\alpha_1$ and $\alpha_2$ such that
$(T_{\alpha}T^{-1})(U)\subset V$ for each $\alpha\geq\alpha_1$ and
$(S_{\alpha}S^{-1})(U)\subset V$ for each $\alpha\geq\alpha_2$.
Choose $\alpha_0$ with $\alpha_0\geq\alpha_1$ and
$\alpha_0\geq\alpha_2$. If $\alpha\geq\alpha_0$, then we have
\[
(T_{\alpha}S_{\alpha})(TS)^{-1}(U)\subset
(T_{\alpha}T^{-1})(U)(S_{\alpha}S^{-1})(U)\subset VV\subset W.
\]
Also, because of continuity of the inversion we have
\[
({T_{\alpha}}T^{-1})^{-1}(U)=(T_{\alpha}T^{-1}(U))^{-1}\subset W,
\]
for sufficiently large $\alpha$. This completes the proof.
\end{proof}

\begin{theorem}
The operations of multiplication and inversion are continuous in
${\sf Hom_{bb}}(G)$ with respect to the topology of uniform
convergence on bounded sets.
\end{theorem}
\begin{proof}
Suppose $(T_{\alpha})$ and $(S_{\alpha})$ are two nets of ${\sf
bb}$-bounded homomorphisms which are convergent uniformly on bounded
sets to the ${\sf bb}$-bounded homomorphisms $T$ and $S$,
respectively. Fix a bounded set $B\subset G$. Let $W$ be an
arbitrary neighborhood of $e_G$. There exists a neighborhood $V$ of
$e_G$ with $VV\subset W$. There are some $\alpha_1$ and $\alpha_2$
such that $(T_{\alpha}T^{-1})(B)\subset V$ for each
$\alpha\geq\alpha_1$ and $(S_{\alpha}S^{-1})(B)\subset V$ for each
$\alpha\geq\alpha_2$. Choose $\alpha_0$ with $\alpha_0\geq\alpha_1$
and $\alpha_0\geq\alpha_2$. If $\alpha\geq\alpha_0$, then we have
\[
(T_{\alpha}S_{\alpha})(TS)^{-1}(B)\subset (T_{\alpha}T^{-1})(B)
(S_{\alpha}S^{-1})(B) \subset VV\subset W.
\]
Since the inversion in $G$ is continuous, for sufficiently large
$\alpha$ we have
\[
({T_{\alpha}}T^{-1})^{-1}(B)=(T_{\alpha}T^{-1}(B))^{-1}\subset W
\]
which completes the proof.
\end{proof}

\begin{theorem}
The operations of multiplication and inversion are continuous in
${\sf Hom_{c}}(G)$, i.e. ${\sf Hom_{c}}(G)$ is a topological group.
\end{theorem}
\begin{proof}
Suppose $(T_{\alpha})$ and $(S_{\alpha})$ are two nets of continuous
homomorphisms $c$-converging to the homomorphisms $T$ and $S$,
respectively. Let $W$ and $V$ be arbitrary neighborhoods of $e_G$.
There exist neighborhoods $W_1$ and $V_1$ of $e_G$ such that
$W_1W_1\subset W$ and $V_1V_1\subset V$. There are a neighborhood
$U$ of $e_G$ and some indices $\alpha_1$ and $\alpha_2$ with
$(T_{\alpha}T^{-1})(U)\subset V_1W_1$ for each $\alpha\geq\alpha_1$
and $(S_{\alpha}S^{-1})(U)\subset V_2W_2$ for each
$\alpha\geq\alpha_2$. Choose $\alpha_0$ such that
$\alpha_0\geq\alpha_1$ and $\alpha_0\geq\alpha_2$. If
$\alpha\geq\alpha_0$, then we have
\[
(T_{\alpha}S_{\alpha})(TS)^{-1}(U)\subset (T_{\alpha}T^{-1})(U)
(S_{\alpha}S^{-1})(U)\subset V_1W_1V_2W_2\subset VW.
\]
Now, continuity of the inversion implies
\[
({T_{\alpha}}T^{-1})^{-1}(U)=(T_{\alpha}T^{-1}(U))^{-1}\subset VW,
\]
for sufficiently large $\alpha$. This completes the proof.
\end{proof}

\medskip
In this part, we investigate whether or not each class of bounded
homomorphisms in the assumed topology is uniformly complete.

\begin{lemma}\label{3}
Suppose $(S_{\alpha})$ is a net of continuous homomorphisms which
converges to the homomorphism $S$ in the ${\sf c}$-convergence
topology. Then, $S$ is also continuous.
\end{lemma}
\begin{proof}
Let an arbitrary neighborhood $W$  of $e_G$ be given. Choose a
neighborhood $V$ of $e_G$ such that $V^3\subset W$. There are a
neighborhood $U$ of $e_G$ and an $\alpha_0$ with
$(S_{\alpha}S^{-1})(U)\subset V^2$ for each $\alpha\ge\alpha_0$. Fix
an $\alpha\ge\alpha_0$. There exists a neighborhood $U_1\subset U$
such that $S_{\alpha}(U_1)\subset V$. From here, together with
$(S_{\alpha}S^{-1})(U_1)\subset V^2$, it follows
\[
S(U_1)\subset S_{\alpha}(U_1) V^2\subset V^3 \subset W,
\]
as desired.
\end{proof}

\begin{lemma}\label{2}
If $(S_{\alpha})$ is a net of ${\sf bb}$-bounded homomorphisms which
converges to the homomorphism $S$ in the topology of uniform
convergence on bounded sets. Then $S$ is {\sf bb}-bounded.
\end{lemma}
\begin{proof}
Let $W$ be an arbitrary neighborhood of $e_G$. Fix a bounded set $B
\subset G$. There is an $\alpha_0$ such that
$(S_{\alpha}S^{-1})(B)\subset W$ for each $\alpha\ge\alpha_0$. Fix
an $\alpha\ge\alpha_0$. There exists a natural $n$  with
$S_{\alpha}(B)\subset W^n$, so that
\[
S(B)\subset (S_{\alpha}(B))W\subset W^n W= W^{n+1},
\]
as we wanted.
\end{proof}

\begin{remark} \rm \label{4}
The class ${\sf Hom_{nb}}(G)$ can contain a Cauchy sequence whose
limit is not an $nb$-bounded homomorphism; in other words, ${\sf
B_{nb}}(G)$ is not uniformly complete in the assumed topology. Let
$G$ be as in Example \ref{ex1} and $(S_n)$ be a sequence of
homomorphisms on $G$ which are defined as follows:
\[
S_n((x_n))=(x_1,\ldots,x_n,1,\ldots).
\]
Each $S_n$ is ${\sf nb}$-bounded. For, if $U_n$ is the neighborhood
of $e_G$ defined by
\[
U_n=\{(x_n), |x_i-1|\le \frac{1}{2}, i=1,2,\ldots,n\},
\]
then, as it is easy to see, $S_n(U_n)$ is bounded in $G$. On the
other hand, it is not difficult to show that $(S_n)$ is uniformly
convergent to the identity homomorphism $1_G$ on $G$. But we have
seen in Example \ref{ex1} that $1_G$ is not ${\sf nb}$-bounded.
\end{remark}

Now, we are going to find conditions under which each class of
considered bounded homomorphisms is topologically complete. In the
case of bounded operators on topological vector spaces, absorbing
neighborhoods and local convexity are two fruitful tools for
discovering conditions (see \cite{Omid 1}). In the topological group
version, it turns out that boundedness of every singleton is a handy
tool. By \cite[Theorem 7.4]{Hewitt}, when $G$ is a connected
topological group, then it is absorbed by positive powers of any
neighborhood of $e_G$, so that every singleton is bounded. Note that
connectedness is not a necessary condition; for example, the
additive group $\mathbb Q$ of rational numbers is totally
disconnected, but every its singleton is bounded. There are examples
of abelian topological groups whose singletons are not bounded. For
example, let $G$ be an abelian topological group and consider the
topological group $H=G\times {\mathbb Z}_2$. Then $G\times \{0\}$ is
a zero neighborhood which is not absorbing, so that singletons are
not bounded.

In what follows we assume that every singleton in the underlying
topological group $G$ is bounded.

\begin{theorem}
If $G$ is a complete group, then ${\sf Hom_{c}}(G)$ is complete.
\end{theorem}
\begin{proof}
Suppose $(T_{\alpha})$ is a Cauchy net in ${\sf Hom_{c}}(G)$ and $W$
is an arbitrary neighborhood of $e_G$. Choose a neighborhood $V$ of
$e_G$ with $VV\subset W$. Find a neighborhood $U$ of $e_G$ and an
index $\alpha_0$ such that $(T_{\alpha}{T_{\beta}}^{-1}(U))\subset
VV$ for each $\alpha\geq \alpha_0$ and for each $\beta\geq\alpha_0$.
Suppose $x\in G$.

First, assume that $x\in U$. Then $T_{\alpha}{T_{\beta}}^{-1}(x)\in
W$. Thus, $((T_{\alpha}(x))$ is a Cauchy net in $G$.

For an arbitrary $x\in G$, there is a positive integer $n$ such that
$x\in U^n$. Since the product of Cauchy nets in a topological group
is again Cauchy, it follows that $(T_{\alpha}(x))$ is a Cauchy net
in $G$, so that it converges. Put $T(x):=\lim T_{\alpha}(x)$ for
each $x\in G$. Since this convergence holds in ${\sf Hom_{c}}(G)$,
by Lemma \ref{3}, $T$ is also continuous, and this completes the
proof.
\end{proof}

\begin{theorem}
If $G$ is a complete group, then ${\sf Hom_{bb}}(G)$ is also
complete.
\end{theorem}
\begin{proof}
Suppose $(T_{\alpha})$ is a Cauchy net in ${\sf Hom_{bb}}(G)$ and
$W$ is an arbitrary neighborhood of $e_G$. One can find a
neighborhood $V$ of $e_G$ such that $VV \subset W$. Since every
singleton $x\in G$ is bounded, there is an $\alpha_0$ such that
$T_{\alpha}{T_{\beta}}^{-1}(x)\in V$ for each $\alpha\geq \alpha_0$
and each $\beta\geq\alpha_0$. One concludes that $((T_{\alpha}(x))$
is a Cauchy net in $G$. Therefore, it converges. Put $T(x):=\lim
T_{\alpha}(x)$, for each $x\in G$. Now fix a bounded set $B\subset
G$. For sufficiently large $\alpha$ and $\beta$, we have
$T_{\alpha}{T_{\beta}}^{-1}(B)\subset V$, so that for each $x\in B$,
\[
T_{\alpha}{T_{\beta}}^{-1}(x)\in V.
\]
For sufficiently large $\beta$ we have $T_{\beta}(x)T^{-1}(x)\in V$,
and therefore
\[
T_{\alpha}{T}^{-1}(x)=
T_{\alpha}{T_{\beta}}^{-1}(x)T_{\beta}(x){T}^{-1}(x)\in VV \subset
W.
\]
Since this convergence holds in ${\sf Hom_{bb}}(G)$, by Lemma
\ref{2}, $T$ is also ${\sf bb}$-bounded, as required.
\end{proof}

\begin{remark} \rm
Note that when $G$ is a complete group, then ${\sf Hom_{nb}}(G)$
might fail to be a complete topological group. Consider Example
\ref{ex1} and Remark \ref{4}.
\end{remark}

\section{Conclusion}

We considered here two kinds of bounded homomorphisms in topological
groups. It would be interesting to consider other sorts of bounded
homomorphisms as well. Recall that a subset $B$ of a topological
groups $G$ is \emph{$\aleph_0$-bounded} if for each neighborhood $U$
of $e_G$ there is a countable set $C$ such that $B\subset CU$;
$B\subset G$ is \emph{ Menger-bounded} (or shortly \emph{{\sf
M}-bounded}) if for each sequence $(U_n:n\in\mathbb N)$ of
neighborhoods of $e_G$ there is a sequence $(F_n:n\in\mathbb N)$ of
finite subsets of $G$ such that $B \subset \bigcup_{n\in\mathbb
N}F_nU_n$. If in the last definition $F_n$s are singletons, then $B$
is said to be \emph{Rothberger-bounded} (shortly \emph{${\sf
R}$-bounded} (see \cite{coc11, star-survey} for more details about
these classes of sets). A homomorphism $T$ between topological
groups $G$ and $H$ is said to be:
\begin{itemize}
\item[$(i)$] \emph{{\sf $\omega$b}-bounded} if for each
$\aleph_0$-bounded subset $B$ of $G$, $T(B)$ is bounded in $H$;

\item[$(ii)$] \emph{{\sf mb}-bounded} if for each ${\sf M}$-bounded
subset $B$ of $G$, $T(B)$ is bounded in $H$;

\item[$(iii)$] \emph{{\sf rb}-bounded} if for each ${\sf R}$-bounded
subset $B$ of $G$, $T(B)$ is bounded in $H$.
\end{itemize}

\section*{Acknowledgement}

We are grateful to the referees for useful remarks and suggestions.

%%% ENTER REFERENCES IN THE FORM

\end{document}